\documentclass[11pt,reqno]{amsart}
\usepackage[T1]{fontenc}
\usepackage{lmodern,amssymb,mathtools,microtype}
\usepackage[margin=1.05in]{geometry}
\usepackage{hyperref}
\newtheorem{theorem}{Theorem}[section]
\newtheorem{lemma}[theorem]{Lemma}
\newtheorem{proposition}[theorem]{Proposition}
\newtheorem{corollary}[theorem]{Corollary}
\newtheorem{Question}[theorem]{Question}

\theoremstyle{remark}
\newcommand{\R}{\mathbb R}
\newcommand{\T}{\mathbb T}\newcommand{\N}{\mathbb N}\newcommand{\Z}{\mathbb Z}
\newcommand{\sa}{\mathrm{sa}}\newcommand{\norm}[1]{\left\lVert#1\right\rVert}
\newcommand{\dist}{\operatorname{dist}}\newcommand{\FS}{\operatorname{FS}_{\sa}}
\title[Single generation theorem]{A single generation theorem for simple unital $C^*$-algebras with CPoU and comparison}

\author{Xin Ma}
	
	\address{Institute for Advanced Study in Mathematics, Harbin Institute of Technology, Harbin, China, 150001}
  \email{xma17@hit.edu.cn}

\date{\today}
\begin{document}
\begin{abstract}
In this note,  assuming the finite factor generation theorem obtained by OpenAI, we demonstrate that if a
separable simple unital $C^*$-algebra $A$ with nonempty trace space, has the strict 
comparison for positive elements, and its uniform tracial completion has complemented partitions of unity (CPoU), then $A$ is singly generated. 
\end{abstract}
\maketitle

\section{Introduction}
Most recently, OpenAI has announced a solution to the single generation problem for $\textrm{II}_1$ factors in \cite{SG}.  In particular, the following was proved in \cite{SG}.
\begin{theorem}(\cite[Theorem 1.2]{SG})\label{thm:SG}
    Every $\operatorname{II}_1$ factor with separable predual is generated, as a von Neumann algebra, by two self-adjoint elements.
\end{theorem}

The analogous single generation question for $C^*$-algebras have also been studied for a long time. In particular, the following question was formulated explicitly in the literature.

\begin{Question}(\cite[Question 1.2]{TW} and \cite[Problem XCVI]{STW})
    Is every unital simple separable $C^*$-algebras singly generated?
\end{Question}
However, it was demonstrated in \cite{ELN1} a simple unital AH counterexample to this question. Nevertheless, in many cases, the $C^*$-algebras are now known to be singly generated. We refer to, e.g., \cite{BCC}, \cite{ELN}, \cite{LNR}, \cite{OZ}, \cite{TW}, and \cite{T} for some results in the positive direction. Many regularity properties of $C^*$-algebras, such as $\mathcal Z$-stability and strict comparison for positive elements, have been shown playing essential roles in the single generation result.

Motivated by recent progress of \cite{ELN} and \cite{SG}, the following single generation theorem is obtained in this paper under the assumption of Theorem \ref{thm:SG}.

\begin{theorem}\label{thm:main}
Let $A$ be separable,
simple, and unital with $T(A)\ne\varnothing$. Suppose
\begin{equation}\label{eq:comparison}
 a,b\in A_+,\qquad d_\rho(a)<d_\rho(b)\quad(\rho\in T(A))
 \quad\Longrightarrow\quad a\precsim b.
\end{equation}
If the uniform tracial completion of $A$ has complemented partitions
of unity (CPoU), then $A$ is singly generated.
\end{theorem}

This result can be applied to a large class of unital simple separable $C^*$-algebras $A$. For example, it follows from \cite[Theorem 1.4 and Proposition 6.3]{TC} that if $A$ has the uniform property $\Gamma$ or the extreme boundary of $T(A)$ is compact and zero-dimensional, then the uniform tracial completion of $A$ has CPoU.

\begin{corollary}\label{cor:zero-dim}
Let $A$ be separable,
simple, and unital $C^*$-algebra satisfying the comparison \eqref{eq:comparison} .  Suppose the extreme boundary $\partial_eT(A)$ is non-empty, compact, and of zero-dimensional,  or suppose $A$ has the uniform property $\Gamma$.Then $A$ is singly generated.
\end{corollary}

Corollary \ref{cor:zero-dim} applies particularly to selfless $C^*$-algebras, which are originally introduced in \cite{R}. It follows from \cite[Theorem 3.1]{R} that a tracial (unital) separable selfless $C^*$-algebra $A$ is simple and has a unique trace and strict comparison for positive elements. This then applies to many reduced group $C^*$-algebras such as acylindrically hyperbolic groups with trivial amenable radicals (see \cite{AGEP}, \cite{O}, and \cite{Y}) and many groups arising from Bass-Serre theory including all non-amenable generalized Baumslag-Solitar (GBS) groups with trivial amenable radical (see more in \cite{MWY}). We warn that this is not a complete list of known $C^*$-selfless groups.

\begin{corollary}\label{cor:selfless}
    All unital separable tracial selfless $C^*$-algebras are singly generated. In particular, all selfless reduced group $C^*$-algebras are singly generated.
\end{corollary}

\subsection*{Outline of the paper:} 
The paper is organized as follows. We adopt a strategy similar to that in \cite{ELN}. Section \ref{sec:all-traces} collects technical results showing how to approximate an element of the $C^*$-algebra by polynomials in the tracial sense. The approximation is first carried out at each extreme trace; then integration over the extreme boundary together with CPoU yields a uniform approximation. In Section \ref{sec:final}, the Baire category method is used to choose a unitary with full spectrum and to verify the hypotheses of \cite[Theorem 5.1]{ELN}, which employs uniform tracial approximation and comparison to produce a single generator for the $C^*$-algebra. The Baire spaces used differ according to whether the $C^*$-algebra has real rank zero.

\section{Local and Uniform approximations in traces }\label{sec:all-traces}
The following is a direct application of \cite[Lemma 2.9]{ELN} based on the generation theorem for $\operatorname{II}_1$ factor with separable predual in Theorem \ref{thm:SG}.

\begin{lemma}(\cite[Lemma 2.9]{ELN})\label{lem:ELN1}
    Let $M$ be a finite factor with separable predual. Let $\alpha\in \mathbb R$, $c\in M_\sa$, and $\rho>0$. Then there exists $H, K\in M_\sa$ such that
    \begin{enumerate}
        \item $H=\sum_{i=1}^m\lambda_ie_i$, where each $\lambda_i\in \R$, and the projections $e_i$ are pairwise orthogonal and $\sum_{i=1}^me_i=1_M$;
        \item $\|H-1_M\|<\rho$ and $\|K-c\|<\rho$;
        \item $W^*(H, K)=M$.
     \end{enumerate}
\end{lemma}

We use the functions $\psi_j$ of \cite[Section 5.2]{ELN},
$j=0,\ldots,22$ defined as follows. On $\T=\R/\Z$, first set
$\rho_j(s)=23\dist(s,j/23)$. Then define functions $\psi_j$ by $\psi_j=1$ whenever
$\rho_j\leq13/20$ and $\psi_j=0$ whenever $\rho_j\geq4/5$, and
interpolate linearly. Then define
\begin{equation}\label{eq:cover}
 U_j=\{z\in\T:\psi_j(z)>1/2\},\qquad j=0,\dots, 22.
\end{equation}
Note that these $U_j$ form an open cover of $\T$. Indeed, for any $s\in \T$, there always exists a $j$ such that $\dist(s, j/23)<1/46$, and therefore one has $\rho_j(s)\leq 1/2$, in which case, $\psi_j(s)=1$.

For fixed $t>0$, define
\[
 u(a)=e^{ita},\qquad
 Y(a,\mathbf b)=\sum_{j=0}^{22}\psi_j(u(a))b_j\psi_j(u(a)).
\]

We now prove our first technical lemma, which is based on Lemma \ref{lem:ELN1} and thus is based on Theorem \ref{thm:SG}.
 
\begin{lemma}\label{lem:factor-density}
Let $M$ be a $\mathrm{II}_1$ factor with separable predual. Suppose
$a\in M_{\sa}$ satisfies $\sigma(e^{ita})=\T$, and
$\mathbf b\in M_{\sa}^{23}$. For every $\varepsilon>0$ there are
$a'\in M_{\sa}$ and $\mathbf b'\in M_{\sa}^{23}$ such that
\[
 \norm{a'-a}<\varepsilon,\qquad
 \max_j\norm{b'_j-b_j}<\varepsilon,
 \qquad W^*(u(a'),Y(a',\mathbf b'))=M.
\]
One may take $a'$ to have finite spectrum. 
\end{lemma}

\begin{proof}
Let $a\in M_\sa$ such that $\sigma(e^{ita})=\T$. Using  Borel functional calculus, one may approximate $a$ by  $a^{(0)}=\sum_{i=1}^q\alpha_i p_i$ such that
\[
\|a-a^{(0)}\|<\varepsilon/2.
\]
where the $p_i$ are nonzero orthogonal projections and $\sum_{i=1}^qp_i=1$. In addition, one may assume all these $z_i=e^{it\alpha_i}$ are distinct by small perturbations. Moreover, choose a Lebesgue number for the open cover obtained by all $U_j$ defined in \eqref{eq:cover}. Since $\sigma(e^{ita})=\T$, by ordering all the $z_i$ cyclically, one may further arrange $z_i$ such that every
consecutive pair $z_i,z_{i+1}$ lies in a common $U_j$. Choose $j(i)$ with $z_i\in U_{j(i)}$ and, for $i<q$, choose
$h(i)$ with $z_i,z_{i+1}\in U_{h(i)}$. Then,
there are pairwise disjoint small phase arcs $V_i$ around $z_i$ such
that
\begin{equation}\label{eq:phase-arcs}
 V_i\subset U_{j(i)},\qquad
 V_i\cup V_{i+1}\subset U_{h(i)}\quad(i<q).
\end{equation}
We start from $u^{(0)}=e^{ita^{(0)}}$ and
$Y^{(0)}=\sum_j\psi_j(u^{(0)})b_j\psi_j(u^{(0)})$. Inductively we will define new $\mathbf{b}'$ and $Y'$.

For each $i$, It follows from Lemma \ref{lem:ELN1} that in $p_iMp_i$, for $\alpha_i p_i$ and $p_iY^{(0)}p_i$, there exits a finite-spectrum $H_i=\sum_{r=1}^{m_i}\lambda_{ir}e_{ir}$ with $e_{ir}$ pairwise orthogonal and $\sum_{r=1}^{m_i}e_{ir}=p_i$,
and $K_i=K_i^*$ such that
\begin{enumerate}
    \item\label{cond} $\|H_i-\alpha_ip_i\|<\varepsilon/2$, and $e^{it\sigma(H_i)}\subset V_i$, and the exponential is injective on the short real interval $I_i$ containing
$\sigma(H_i)$;
    \item $\|K_i-p_iY^{(0)}p_i\|<\varepsilon/32$;
    \item $W^*(H_i,K_i)=p_iMp_i$.
\end{enumerate}
Set
\[
 a'=\sum_{i=1}^qH_i=\sum_{i=1}^q\sum_{r=1}^{m_i}\lambda_{ir}e_{ir},\quad u'=e^{ita'}=\sum_{i}\sum_re^{it\lambda_{ir}}e_{ir},\quad
 \widehat Y=\sum_j\psi_j(u')b_j\psi_j(u').
\]
Because all $p_i$ are distinct, one actually has $e_{ir}$ are pairwise orthogonal to each other. Then define $\xi_{ir}=e^{it\lambda_{ir}}$. Recall that $e^{it\sigma(H_i)}\subset V_i$, which are disjoint from each other.  Moreover, all $\lambda_{ir}\in I_i$ and the function $x\mapsto e^{itx}$ is injective on $I_i$. These imply that all the $\xi_{ir}$ are different. For each $1\leq i\leq q$ and $1\leq r\leq m_i$, we define a function
\[L_{ir}(z)=\prod_{(j,s)\neq(i,r)}\frac{z-\xi_{js}}{\xi_{ir}-\xi_{js}}.\]
This yields $L_{ir}(\xi_{js})=1$ if $(j,s)=(i,r)$ and $L_{ir}(\xi_{js})=0$ if $(j,s)\neq(i,r)$. Therefore, one has
\[L_{ir}(u')=\sum_{j,s}L_{ir}(\xi_{js})e_{js}=e_{ir},\]
which implies that $p_i, H_i\in C^*(u')$.
Moreover, note that by construction 
\[\|a'-a^{(0)}\|=\max_i\|H_i-\alpha_ip_i\|<\varepsilon/2.\]
Moreover, if necessary, we may make $H_i$ more closer to $\alpha_ip_i$ such that
$\|\widehat Y-Y^{(0)}\|<\varepsilon/32$, which implies that $\|K_i-p_i\widehat Yp_i\|\leq \epsilon/32$.

We now work in the $p_i$ corner $p_iMp_i$. Note that $\psi_{j(i)}(u')=\sum\psi_{j(i)}(\xi_{js})e_{js}$ and therefore $p_i\psi_{j(i)}(u')p_i=\sum_{r=1}^{m_i}\psi_{j(i)}(\xi_{ir})e_{ir}$, which is invertible in $p_iMp_i$ with the inverse $Q_i=\sum\psi^{-1}_{j(i)}(\xi_{ir})e_{ir}$ such that $\|Q_i\|\leq 2$ as each $\psi_{j(i)}(\xi_{ir})>1/2$ by \eqref{eq:phase-arcs}. Define pairwise orthogonal elements
\[
 D_i=Q_i(K_i-p_i\widehat Yp_i)Q_i\in (p_iMp_i)_\sa.
\]
Add $D_i$ to coordinate
$b_{j(i)}$ to define
\[
 Y^{(1)}=\widehat Y+\sum_i(K_i-p_i\widehat Yp_i), 
\]
which satisfies $p_iY^{(1)}p_i=K_i$.

We now deal with $p_iY^{(1)}p_{i+1}$ for each $i\leq q-1$. If $p_iY^{(1)}p_{i+1}\ne0$, we do nothing.
If $p_iY^{(1)}p_{i+1}=0$, choose a nonzero
$s_i\in p_iMp_{i+1}\neq 0$ with $\|s_i\|\leq \varepsilon/32q$. For $h(i)$, in the similar way, $\psi_{h(i)}(\xi_{ir})>1/2$ and $\psi_{h_i}(\xi_{(i+1)r})>1/2$ by \eqref{eq:phase-arcs} implies that
 $A_i=p_i\psi_h(u')p_i$ and $B_i=p_{i+1}\psi_h(u')p_{i+1}$ are invertible in $p_iMp_i$ and $p_{i+1}Mp_{i+1}$, respectively, with norm $\|A_i\|, \|B_i\|\leq 2$. Define $T_i=A^{-1}_is_iB^{-1}_i$ and we add to coordinate
$b_{h(i)}$ by $T_i+T_i^*$.  This implies that
\[\psi_{h(i)}(u')(T_i+T_i^*)\psi_{h(i)}(u')=s_i+s_i^*.\]
Denote the final tuple by $\mathbf b'$ and its
localized sum by $Y'=Y'(a',\mathbf{b}')=\sum_j\psi_j(u')b'_j\psi_j(u')$. Note by our construction, for each $j$, the difference between $b_j$ and $b'_j$ are some $D_i$ and some $T_i+T^*_i$. Since all $D_i$ are orthogonal, one actually has
\[\|\sum_i D_i\|=\max_i\|D_i\|\leq 4\max_i\|K_i-p_i\widehat Yp_i\|\leq \varepsilon/8.\]
Moreover, each $\|T_i+T^*_i\|\leq 8\|s_i\|$ and thus 
\[\sum_{1\leq i\leq q,\ p_iY^{(1)}p_{i+1}=0}\|T_i+T^*_i\|\leq \varepsilon/4.\]
This implies that $\max_{j}\|b'_j-b_j\|<\varepsilon/2$. In addition, since  $\norm{a^{(0)}-a}<\varepsilon/2$ and
$\norm{a'-a^{(0)}}<\varepsilon/2$, one has
$\norm{a'-a}<\varepsilon$. 

Finally note that $p_iY'p_i=K_i$ and $p_iY'p_{i+1}\neq 0$.
The algebra $W^*(u',Y')$ contains $p_i,H_i,K_i$. By our choice, $p_iMp_i=W^*(H_i, K_i)$, which implies that $W^*(u', Y')$ contains all full
corners $p_iMp_i$. Moreover,  it contains a nonzero element of every
$p_iMp_{i+1}$. Then it follows from \cite[Lemma 2.8]{ELN} that $W^*(u',Y')=M$.
\end{proof}

The distinction between extreme and arbitrary traces is essential. Next lemma deal with the approximation with respect to a general trace.

\begin{lemma}\label{lem:all-traces}
Let $A$ be separable, simple, unital, and infinite
dimensional with $X=T(A)\ne\varnothing$.
Suppose $a=a^*\in A$ has $\sigma(e^{ita})=\T$ and
$\mathbf b\in A_{\sa}^{23}$. For any $\rho\in X$, self-adjoint
contraction $z\in A$, and $r,\varepsilon>0$, there are parameters
$(a_\rho,\mathbf b_\rho)$ within $r$ in each coordinate of
$(a,\mathbf b)$ and a self-adjoint $*$-polynomial $P_\rho$ with
\[
 \norm{z-P_\rho(u(a_\rho),Y(a_\rho,\mathbf b_\rho))}_{2,\rho}
 <\varepsilon.
\]
\end{lemma}
\begin{proof}
For each $\tau\in\partial_eX$, denoting by $\pi=\pi_\tau$ for simplicity, it is standard that $M_\tau=\pi(A)''$ is a
$\mathrm{II}_1$ factor with separable predual. Recall that $\tau$ induces a canonical faithful normal trace $\hat{\tau}$ on $M$. For simplicity, we denote by $\|x\|_2=\hat{\tau}(x^*x)^{1/2}$ for $x\in M$. In this way, note that $\|\pi(x)\|_2=\|x\|_{2,\tau}$ for $x\in A$.

The simplicity of $A$ makes
$\pi_\tau$ faithful, so $\sigma(e^{it\pi(a)})=\T$. Now for the $z\in A$, it follows from Lemma \ref{lem:factor-density} for $\pi(a), \pi(b_0),\dots,\pi(b_{22})$ and $r>0$ that there exists a $\tilde{a}$ with finite spectrum, and $\tilde{b_1},\dots, \tilde{b_{22}}$ such that
$\|\tilde{a}-\pi(a)\|<r/16$, $\max_j\|\tilde{b_j}-\pi(b_j)\|<r/16$, and by setting $\tilde{u}=e^{it\tilde{a}}$ and $\tilde{y}=\sum_{j=0}^{22}\psi_j(\tilde{u})\tilde{b_j}\psi_j(\tilde{u})$, one has $W^*(\tilde{u}, \tilde{y})=M$.
Thus, $C^*(1, \tilde{u}, \tilde{y})$ is strongly dense in $M$, one may assume there exists a $*$-polynomial $P$ such that $\|\pi(z)-P(\tilde{u}, \tilde{y})\|_2<\varepsilon/32$. Since $c=c^*$, one may assume $P=P^*$.

Write $d=\tilde{a}-\pi(a)$, and $e_j=\tilde{b_j}-\pi(b_j)$ for simplicity. Then $\|d\|, \|e_j\|<r/16$ for any $j$. Then for any $\eta>0$, Kaplansky density theorem implies that there exists self-adjoint elements $d_\eta, e_{j, \eta}\in A$ such that $\|d_\eta\|\leq r/16$, $\|e_{j, \eta}\|\leq r/16$, as well as $\|\pi(d_\eta)-d\|_2<\eta$ and $\|\pi(e_{j, \eta})-e_j\|_2<\eta$ for any $j$. Now define $a_\eta=a+d_\eta$ and $b_{j, \eta}=b_j+e_{j,\eta}$. Then one has 
\[\|a_\eta-a\|<\eta/16, \qquad \|b_{j,\eta}-b_j\|<\eta/16\]
and
\[\|\pi(a_\eta)-\tilde{a}\|_2<\eta, \qquad \|\pi(b_{j,\eta})-\tilde{b_j}\|_2<\eta.\]
Define $u_\eta=e^{it\pi(a_\eta)}$. By continuity,  one has
\[\|u_\eta-\tilde{u}\|<t\eta\]
and there exists a constant $L_0$, depending on $\psi_j$ for $j=0,\dots, 22$, such that
\[\|\psi_j(u_\eta)-\psi_j(\tilde{u})\|<L_0t\eta.\]
Define $y_\eta=\sum_{j=0}^{22}\psi_j(u_\eta)\pi(b_{j, \eta})\psi_j(u_\eta)$ and $K=\max_{0\leq j\leq 22}\|b_j\|+r/16$. This implies that 
\[\|y_\eta-\tilde{y}\|_2<(46KL_0t+23)\eta\]
and therefore there exists constant $\Lambda_P, C_0>0$ such that 
\[\|P(u_\eta, y_\eta)-P(\tilde{u}, \tilde{y})\|_2\leq \Lambda_PC_0\eta.\]
Therefore, taking $0<\eta<\varepsilon/32C_0(1+\Lambda_P)$ one has
\[\|z-P(u(a_\eta), Y(a_\eta, \mathbf b_\eta))\|_{2,\tau}=\|\pi(z)-P(u_\eta, y_\eta)\|_2<\varepsilon/16.\]
Write $C_\tau=P(u(a_\eta), Y(a_\eta, \mathbf b_\eta)\in A$ for simplicity and define an open set $O_\tau=\{\sigma\in T(A):\sigma((z-C_\tau)^*(z-C_\tau))<\varepsilon^2/64\}$ such that $\|z-C_\tau\|_{2, \sigma}<\varepsilon/8$. Since $X=T(A)$ is second countable, one may choose countably many such $O_l$, which forms an open cover of $\partial_e X$ such that
\begin{equation}\label{eq:local-extreme}
 C_l=P_l(u(a_l),Y(a_l,\mathbf b_l)),\qquad
 \norm{z-C_l}_{2,\tau}<\varepsilon/8\quad(\tau\in O_l).
\end{equation}
in which $\|a-a_l\|<r/8$ and $\max_j\|b_{l, j}-b_j\|<r/8$.

Now let $\rho\in X$ be a general element and let $\mu_\rho$ be the representing measure of $\rho$ on $X$ with the support $\partial_eX$.
Choose finitely many of the above neighborhoods so that
\[
 \mu_\rho\Bigl(\bigcup_{l=1}^LO_l\Bigr)>1-(\varepsilon/4)^2.
\]
Then define $V_1=O_1\cap \partial_eX$, and $V_l=(O_l\cap \partial_eX)\setminus \bigcup_{k<l}V_k$.
These are disjoint measurable sets $V_l\subset O_l\cap\partial_eX$ and satisfies
$\bigsqcup_{l=1}^LV_l=\partial_eX\cap (\bigcup_{l=1}^LO_l)$. Define $V_0=X\setminus \bigsqcup_{l=1}^LV_l$.

By \cite[Lemma 3.38]{TC}, there is a normal isomorphism
\[
 \theta_\rho:L^\infty(X,\mu_\rho)\longrightarrow
 Z(M_\rho),\qquad M_\rho=\pi_\rho(A)'',
\]
with
\begin{equation}\label{eq:central-measure}
 \widehat\rho(\theta_\rho(f)\pi_\rho(x))
 =\int_Xf(\tau)\tau(x)\,d\mu_\rho(\tau).
\end{equation}
in which $\hat{\rho}$ is the canonical faithful normal trace.
Set $p_l=\theta_\rho(1_{V_l})$ for $l=0,1,\dots, L$ which are central
orthogonal projections such that $\sum_{l=0}^Lp_l=1$. Moreover,  \eqref{eq:central-measure} shows that
\begin{align*}
    \|\hat{\rho}(p_l\pi_\rho(z-C_l))\|^2_{2,\hat\rho}&=\hat{\rho}((p_l\pi_\rho(z-C_l))^* p_l\pi_\rho(z-C_l))=\hat{\rho}(p_l\pi_\rho(z-C_l)^*\pi_\rho(z-C_l))\\
    &=\int_{V_i}\tau(\pi_\rho(z-C_l)^*\pi_\rho(z-C_l))d\mu_\rho(\tau)=\int_{V_i}\|z-C_l\|^2_{2,\tau}d\mu_\rho(\tau).
\end{align*}
Combining with \eqref{eq:local-extreme},  this further implies that
\begin{equation}\label{eq:approx1}
    \|p_l\pi_\rho(z-C_l)\|_{2,\hat\rho}\leq (\varepsilon/8)\mu_\rho(V_l)^{1/2}.
\end{equation}
On the other hand, since $z$ is a contraction, one also has
\[\|p_0\pi_\rho(z)\|^2_{2, \hat\rho}=\hat\rho(p_0\pi_\rho(z^*z))\leq \hat\rho(p_0)=\mu_\rho(V_0)\leq (\varepsilon/4)^2\]
which implies
\begin{equation}\label{eq:approx2}
  \|p_0\pi_\rho(z)\|_{2, \hat\rho}\leq \varepsilon/4.
\end{equation}

For $l\geq 0$, approximate $\pi_\rho(a_l)$ in norm by a finite-spectrum
$H_l\in(M_\rho)_{\sa}$ such that
\[\|H_l-\pi_\rho(a_l)\|<\delta_l\]
is small enough such that
\begin{equation}\label{eq:appro1}
    \norm{H_l-\pi_\rho(a)}<r/4
\end{equation}
and for $l\geq 1$,
\begin{equation}\label{eq:appro2}
  \|P_l(W_l, Z_l)-\pi_\rho(C_l)\|\leq \varepsilon/8  
\end{equation}
for $C_l$ in \eqref{eq:local-extreme}, where 
$W_l=e^{itH_l}$, and $Z_l=\sum_j\psi_j(W_l)\pi_\rho(b_{l,j})\psi_j(W_l).$
Furthermore, by small perturbation, one may also arrange that $E_l=\sigma(e^{itH_l})$ for $l\geq 0$ are pairwise disjoint because $H_l$ is of finite spectrum. Define
\[
 H=\sum_{l=0}^Lp_lH_l,\qquad
 B_j=\sum_{l=0}^Lp_l\pi_\rho(b_{l,j}),\qquad
 W=e^{itH},\quad Z=\sum_j\psi_j(W)B_j\psi_j(W).
\]
Since $H-\pi_\rho(a)=\sum_{l}p_l(H_l-\pi_\rho(a))$, one actually has 
\[\|H-\pi_\rho(a)\|\leq \max_l\|H_l-\pi_\rho(a)\|\leq r/4.\]
Similarly, one has $\|B_j-\pi_\rho(b_j)\|\leq r/8$. Using the orthogonality of these $p_l$ again, a direct calculation shows $W=\sum_{l=0}^Lp_lW_l$ and $Z=\sum_{l=0}^Lp_lZ_l$. Furthermore, one has $p_lP(W, Z)=p_l(W_l, Z_l)$.

Define $E=\bigsqcup_{l=0}^L E_l$. For each $\xi\in E$, define function $L_\xi$ by
\[L_\xi(w)=\prod_{\eta\in E,\eta\ne \xi}\frac{w-\eta}{\xi-\eta},\]
which satisfies $L_\xi(\eta)=1$ if $\eta=\xi$ and $L_\xi(\eta)=0$ if $\eta\neq \xi$. Moreover, for each $l$, choose the Laurent polynomial $q_l(w)=\sum_{\xi\in E_l}|L_\xi(w)|^2$ on the circle, which satisfies $q_l|_{E_l}=1$ and $q_l|_{E_k}=0$ for $k\ne l$. This implies that $q_l(W_k)=1$ only if $k=l$. Otherwise $q_l(W_k)=0$. Therefore, one has
\begin{equation}\label{eq:appro3}
    q_l(W)=\sum_{k=0}^Lp_kq_l(W_k)=p_l,
\end{equation}
which implies that $C^*(W)$ contains all these projections $p_l$ for $l=0,\dots, L$. We define polynomial $P_0=0$ and define
a self-adjoint polynomial
\[
 Q(v,w)=\frac12\sum_{l=0}^L
 \bigl(q_l(v)P_l(v,w)+P_l(v,w)q_l(v)\bigr)
\]
satisfies $Q(W,Z)=\sum_lp_lP_l(W_l,Z_l)$ by a direct calculation. Since $P_0=0$, one has $p_0Q(W, Z)=0$.   For $l\geq 1$, one has
\[p_l(\pi_\rho(z)-Q(W,Z))=p_l(\pi_\rho(z)-P_l(W_l, Z_l))=p_l\pi_\rho(z-C_l)+p_l(\pi_\rho(C_l)-P_l(W_l, Z_l)).\]
Then, using \eqref{eq:approx1} and \eqref{eq:appro2}, one has 
\begin{equation}\label{eq:appro5}
    \|p_l(\pi_\rho(z)-Q(W,Z))\|_2\leq (\varepsilon/8)\mu_\rho(V_l)^{1/2}
\end{equation}
for $l\geq 1$. Then for $l=0$, note that
\begin{equation}\label{eq:appro6}
   \|p_0(\pi_\rho(z)-Q(W, Z))\|_2^2=\|p_0\pi_\rho(z)\|_2^2=\hat\rho(p_0\pi_\rho(z^*z))\leq \hat\rho(p_0)=\mu_\rho(V_0)<\varepsilon^2/16. 
\end{equation}
Writing $D=\pi_\rho(z)-Q(W, Z)$ for simplicity, because of $\sum_l p_l=1$, one has
\[\|D\|_2^2=\hat\rho(D^*(\sum_{l=0}^Lp_l)D)=\sum_l\|p_lD\|_2^2.\]
Then it follows from \eqref{eq:appro5} and \eqref{eq:appro6} that
\[\|D\|^2_2\leq (\varepsilon^2/16)\cdot\sum_{l=1}^L\mu_\rho(V_l)+\mu_\rho(V_0)\leq \varepsilon^2/4.\]
Consequently, one has
\begin{equation}
     \norm{\pi_\rho(z)-Q(W,Z)}_{2}\leq \varepsilon/2.
\end{equation}
This is the desired approximation inequality with elements in $M_\rho$. Then using Kaplansky density and continuity. one may approximate $Q(W, Z)$ by elements in $A$ under $2$-norm such that there exists parameters
$(a_\rho,\mathbf b_\rho)$ within $r$ in each coordinate of
$(a,\mathbf b)$ and a self-adjoint $*$-polynomial $P_\rho=Q$ with
\[
 \norm{z-P_\rho(u(a_\rho),Y(a_\rho,\mathbf b_\rho))}_{2,\rho}
 <\varepsilon.
\]
\end{proof}

Now, we will use CPoU to glue the single trace case to uniform case. 
Write $X=T(A)$ and
\[
 \norm{x}_{2,X}=\sup_{\rho\in X}\rho(x^*x)^{1/2}.
\]
Let $\mathcal M$ denote the uniform tracial completion of $A$ relative to $X$. 
For a free ultrafilter $\omega$, identify
\[
 N=\mathcal M^\omega\cong A^\omega
 =\ell^\infty(A)/\{(x_n):\lim_\omega\norm{x_n}_{2,X}=0\}.
\]
Bounded uniform tracial approximation gives this identification.
Write $X^\omega$ for the designated compact set of traces on $N$,
generated by the limit traces. Its uniform $2$-norm satisfies
\[
 \norm{[(x_n)]}_{2,X^\omega}=\lim_\omega\norm{x_n}_{2,X}.
\]
Every $\sigma\in X^\omega$ restricts to a member of $X$ on the
constant copy of $A$.

We need two consequences of CPoU, both from \cite{TC}:
\begin{enumerate}
\item $N$ has real rank zero (\cite[Proposition 7.2]{TC}).
\item Given positive $e_1,\ldots,e_L\in N$, a uniformly $2$-norm
separable $S\subset N$, and
\[
 \delta>\sup_{\sigma\in X^\omega}\min_l\sigma(e_l),
\]
there are orthogonal projections $p_l\in N\cap S'$ with
\begin{equation}\label{eq:cpou}
 \sum_lp_l=1,\qquad \sigma(e_lp_l)\leq\delta\sigma(p_l)
 \quad(\sigma\in X^\omega).
\end{equation}
This is \cite[Proposition 6.2(iii)]{TC}, also expressed by \cite[Corollary 7.8]{TC}.
\end{enumerate}

Then we have the following uniform approximation.

\begin{proposition}\label{prop:uniform-patch}
Under the assumptions of Lemma~\ref{lem:all-traces}, suppose also that
$\mathcal M$ has CPoU. For every self-adjoint contraction $z\in A$
and $r,\varepsilon>0$, there exist parameters $(a_1,\mathbf b_1)$
within $r$ of $(a,\mathbf b)$ in each coordinate and a self-adjoint
$*$-polynomial $Q$ such that
\[
 \norm{z-Q(u(a_1),Y(a_1,\mathbf b_1))}_{2,T(A)}<\varepsilon.
\]
\end{proposition}
\begin{proof}
Still denote by $X$ the whole space $T(A)$ for simplicity. Apply Lemma~\ref{lem:all-traces} at every $\rho\in X$, with parameter
radius $r/8$ and error $\varepsilon/16$ such that there exists $a_\rho, \mathbf b_\rho$ such that
\[\|a-a_\rho\|<r/8, \qquad \sup_j\|b_{\rho, j}-b_j\|<r/8,\]
and a self-adjoint polynomial $P_\rho$ such that $\|z-P_\rho(u(a_\rho), Y(a_\rho,\mathbf b_\rho))\|_{2, \rho}<\varepsilon/16.$ Denote by $C_l=P_\rho(u(a_\rho), Y(a_\rho,\mathbf b_\rho))\in A$ for simplicity and define
\[O_\rho=\{\tau\in X: \|z-C_l\|_{2,\tau}<\varepsilon/8\},\]
which is an open neighborhood of $\rho$. Then the compactness of $X$ allows to choose a finite subcover of these $O_\rho$, which yields $(a_l, \mathbf b_l, P_l, C_l)$ for $1\leq l\leq L$. This leads to
\begin{equation}\label{eq:full-cover}
 \max_{\rho\in X}\min_l\norm{z-C_l}_{2,\rho}<\varepsilon/8.
\end{equation}
Recall $N=A^\omega$ and we regard $A$ as the constant sequences in $N$.
Since $N=A^\omega$ has real rank zero, choose finite-spectrum
$H_l\in N_{\sa}$ such that $\|H_l-a_l\|<r/8$ and thus $\|H_l-a\|<r/4$. Define $V_l=e^{itH_l}$. Then as in Lemma \ref{lem:all-traces}, we may arrange that 
$E_l=\sigma(V_l)$ are pairwise disjoint.
Moreover, writing $Z_l=\sum_j\psi_j(V_l)b_{l,j}\psi_j(V_l)$ and $F_l=P_l(V_l, Z_l)$, one may make the approximations
accurate enough that
\[
 \norm{F_l-C_l}<\varepsilon/8.
\]
We now verify the hypothesis to apply CPoU. Let $\sigma\in X^\omega$.
For $e_l=|z-F_l|^2\in N_+$, one has $\sigma(e_l)=\|z-F_l\|_{2,\sigma}^2$
restriction of a trace
$\sigma\in X^\omega$ to $A$ and \eqref{eq:full-cover} shows
\[
 \sup_{\sigma\in X^\omega}\min_l\sigma(e_l)
 \leq\varepsilon^2/16<\varepsilon^2/4.
\]
Apply CPoU in \eqref{eq:cpou} with $\delta=\varepsilon^2/4$ and with $S=C^*(A, H_1,\dots, H_L)$ to obtain orthogonal projections
$p_l\in S'\cap N$ with $\sum_lp_l=1$ and
\[
\sigma(p_le_l)\leq(\varepsilon^2/4)\sigma(p_l),\qquad (\sigma\in X^{\omega})
\]
Then we mainly follow the argument in Lemma \ref{lem:all-traces}. Define
\[
 H=\sum_lp_lH_l,\quad B_j=\sum_lp_lb_{l,j},\quad
 V=e^{itH}=\sum_lp_lV_l,\quad
 Z=\sum_j\psi_j(V)B_j\psi_j(V)=\sum_lp_lZ_l.
\]
such that $\|H-a\|\leq \max_l\|H_l-a\|<r/4$ and similarly $\|B_j-b_j\|<r/4$.
Also choose the Laurent polynomials $q_l$ equal to one on $E_l$ and zero
on the other $E_m$. Then $q_l(V)=p_l$, so the symmetrized polynomial
\[
 Q(v,w)=\frac12\sum_l(q_l(v)P_l(v,w)+P_l(v,w)q_l(v))
\]
has $Q(V,Z)=\sum_lp_lP_l(V_l,Z_l)$ and
\[
 \norm{z-Q(V,Z)}_{2,\sigma}^{,2}
 =\sum_l\sigma(p_le_l)
 \leq(\varepsilon^2/4)\sum_l\sigma(p_l)=\varepsilon^2/4.
\]
This estimate holds for all $\sigma\in X^\omega$.
Finally, we lift the elements back to $A$.
denote by $\kappa: \ell^\infty(A)\to N$ the usual quotient map.
Lift the self-adjoint differences $D_a=H-a$ and $D_j=B_j-b_j$ to elements $(d_n)$ and $(e_{j,n})$ in $\ell^\infty(A)$ with norms less than $r/2$ using functional calculus.
Then define $x_n=a+d_n$ and $y_{j, n}=b_j+e_{j,n}$. Then it is a direct calculation that $(z-Q(u(x_n),Y(x_n, \mathbf y_n)))$ represents $z-Q(V, Z)$. Then choose $n_0$ large enough such that
\[\|z-Q(u(x_n),Y(x_n, \mathbf y_n)\|_{2, X}<\varepsilon.\]
\end{proof}

\section{Full spectrum and the final norm generator}\label{sec:final}

For a unital $C^*$-algebra $B$, let $\FS(B)$ be its finite-spectrum
self-adjoint elements. For a nonempty compact set $X\subset\R$, put
\[
 L(X)=\sup\{\operatorname{diam}(I): I\text{ is a connected component of }X\}.
\]
Components of $X$ are compact intervals, including singletons.

\begin{lemma}\label{lem:spectral-distance}
For $a=a^*\in B$,
\[
 \dist(a,\FS(B))\leq\tfrac12 L(\sigma(a)).
\]
\end{lemma}

\begin{proof}
Put $X=\sigma(a)$ and fix $\eta>0$. Each component $I=[l,r]$ of $X$
has a clopen neighborhood in $X$ of diameter less than $L(X)+2\eta$.
Indeed, choose $s\notin X$ with $l-\eta<s<l$ and $t\notin X$
with $r<t<r+\eta$, using maximality of the component (and points
outside the whole spectrum at an extreme endpoint). Then
$X\cap(s,t)$ is such a neighborhood.
Compactness gives a finite cover by these clopen sets. Refining the
cover gives a finite clopen partition $X=V_1\sqcup\cdots\sqcup V_q$,
each member still having diameter less than $L(X)+2\eta$.
Define $f=(\max V_i+\min V_i)/2$ on $V_i$.
Then $f\in C(X,\R)$ has finite image and
\[
 \norm{a-f(a)}<\tfrac12L(X)+\eta.
\]
Now, let $\eta$ decrease to zero and then one has $ \dist(a,\FS(B))\leq\tfrac12 L(\sigma(a))$.
\end{proof}

If $B$ does not have real rank zero, the following 
provides a fixed $t>0$ and a nonempty norm-open set
$\Omega\subset A_{\sa}$ consisting of $a$ with
$\sigma(e^{ita})=\T$. 

\begin{corollary}\label{cor:open-spectrum}
If $B$ does not have real rank zero, there exist $t>0$ and a nonempty
norm-open ball $\Omega\subset B_{\sa}$ such that
\[
 \sigma(e^{ita})=\T\qquad(a\in\Omega).
\]
\end{corollary}

\begin{proof}
Since $B$ has no real rank zero, one chooses $a_0\in B_{\sa}$ with
$\delta=\dist(a_0,\FS(B))>0$. Define
\[
 \Omega=\{a\in B_{\sa}:\norm{a-a_0}<\delta/4\},
 \qquad t=2\pi/\delta.
\]
For $a\in\Omega$, the distance to $\FS(B)$ exceeds $3\delta/4$.
Lemma~\ref{lem:spectral-distance} gives
$L(\sigma(a))>3\delta/2$. Thus $\sigma(a)$ has an interval component $[l, r]$
of length greater than $\delta$. This implies that
\[\T\subset \{e^{its}:s\in \sigma(a)\}=\sigma(e^{ita})\subset \T.\]
\end{proof}

For real rank zero, we need the following lemma.
\begin{lemma}\label{lem:repair}
Let $A$ be simple, unital, infinite dimensional, and of real rank zero.
Set $\mathcal S=\{a=a^*\in A:\sigma(e^{ia})=\T\}$.
This is a nonempty norm-closed set. If $a\in\mathcal S$ and
$\norm{\widetilde a-a}<r/8$ for $\widetilde a\in A_\sa$, then, for every $\delta>0$, there exists
$a_2\in\mathcal S$ with
\[
 \norm{a_2-a}<r/2,\qquad \norm{a_2-\widetilde a}_{2,T(A)}<\delta.
\]
\end{lemma}
\begin{proof}
First, since $A$ is simple and of real rank zero,  for any nonzero projection $p\in A$ and any $n\in \N_+$, there exists $q\leq p$ such that $\sup_{\tau\in T(A)}\tau(q)\leq 1/n$. This can be done by looking at $pAp$.

Now, for a non-zero projection $q\in A$ and $R>0$, we will choose a self-adjoint $h\in qAq$ such that $\sigma_{qAq}=[-R, R]$.
We work in $qAq$ with the unit $q$. It allows to write $q=q_0+q_1$ such that $q_0$ and $q_1$ are orthogonal. We recursively do this splitting process for $q_0$ and $q_1$ and so on to obtain a family $\mathcal P=\{p_s: s\in \{0, 1\}^n: n\in \N\}$ of projections such that $\{p_s: |s|=n\}$ is an orthogonal family. This implies that the $C^*$-algebra generated by $\mathcal P$, as a unital $C^*$-subalgebra in $qAq$, is isomorphic to $C(\{0, 1\}^\N)$. Then the function 
\[f(x)=-R+2R\cdot \sum_{k=1}^\infty\frac{x_k}{2^k}\]
corresponds to the desired $h\in qAq$ with $\sigma_{qAq}(h)=\sigma(f)=[-R, R]$.

In particular, for $q=1_A$ and $R=\pi$, there exists $a\in A$ with $\sigma(a)=[-\pi, \pi]$ and thus $\sigma(e^{ia})=\T$. Thus $\mathcal S$ is not empty. Then let $a_n\in \mathcal S$ converges to $a$, a direct calculation shows that $a\in \mathcal S$ as well.

Now we construct the element $a_2$. For the $a, \widetilde a$ be given.
Choose a finite-spectrum $c=\sum_{i=1}^m\alpha_i e_i$ with
\[
 \norm{c-\widetilde a}<\min\{r/8,\delta/2\}.
\]
Then $d=\norm{c-a}<r/4$. Choose $d<R<r/4$ and $n$ large enough such that $R\sqrt{m/n}<\delta/2$. T
In each nonzero $e_i$ choose a nonzero projection $q_i\leq e_i$ such that $\sup_\tau \tau(q_i)\leq 1/n$ and thus
\[
 R\Bigl(\sup_\tau\sum_i\tau(q_i)\Bigr)^{1/2}< R \sqrt{m/n}\leq \delta/2.
\]
Choose $h_i\in(q_iAq_i)_{\sa}$ with spectrum $[-R,R]$. Define $a_2=c+\sum_i h_i$.
Then one has $\norm{a_2-a}\leq R+d<r/2$ and
$\norm{a_2-\widetilde a}_{2,T(A)}<\delta$ because $h_i$ are orthogonal and 
\[\|\sum_i h_i\|^2_{2, T(A)}=\sup_{\tau\in T(A)}\sum_i\tau(h_i^2)\leq \sup_\tau R^2\tau(q_i)\leq R^2m/n.\]
The spectrum of $a_2$ contains every $[\alpha_i-R,\alpha_i+R]$.
Every point of $\sigma(a)$ is within $d<R$ of some $\alpha_i$,
so spectral mapping and $\sigma(e^{ia})=\T$ imply
$\sigma(e^{ia_2})=\T$.
\end{proof}

Now, we are ready to prove our main theorem.
\begin{proof}[Proof of Theorem~\ref{thm:main}]
If $A$ is finite dimensional, simplicity makes it a matrix algebra,
which is singly generated.

In the infinite dimensional case, still write $X=T(A)$. If $A$ has no real rank zero, define $B=\Omega\times A_\sa^{23}$, where $\Omega$ is the nonempty norm-open ball in $A_\sa$ in Corollary \ref{cor:open-spectrum}. Otherwise, Define $B=\mathcal S\times A^{23}_\sa$, where $\mathcal S$ is obtained in Lemma \ref{lem:repair}. In any cases, $B$ is always a Baire space.

By separability of $A$, one may choose a norm-dense sequence $(z_n)$ of self-adjoint contractions in the self-adjoint unit ball of $A$.
For each $n,k$, define an open set  $U_{n,k}$ to be
\[\{(a, \mathbf b)\in B: \text{there exists a self-adjoint}*\text{-polynomial P such that } \norm{z_n-P(u(a),Y(a,\mathbf b))}_{2,X}<1/k\},\]
which is open,  because $(a, \mathbf b)\mapsto  \norm{z_n-P(u(a),Y(a,\mathbf b))}_{2,X}$ is continuous.

Suppose $A$ has no real rank zero. Let $(a, \mathbf b)\in B$, and choose a small $r>0$ such that if $a'$ satisfies $\|a'-a\|<r$, then $a'\in \Omega$.

Then apply Proposition \ref{prop:uniform-patch} to $z_n$, $r$ and $1/k$ to obtain a tuple $(a', \mathbf b')$ and a self-adjoint polynomial $P$ such that 
\[\|z_n-P(u(a'), Y(a', \mathbf b'))\|_{2, X}<1/k.\]
Note that $(a', \mathbf b')\in B$ and this implies that any neighborhood of $(a, \mathbf b)$ meets $U_{n,k}$.

In the real rank zero case, for $(a, \mathbf b)\in B$, apply Proposition \ref{prop:uniform-patch} to $z_n$, $r$ and $1/2k$ to obtain a tuple $(\tilde{a}, \tilde{\mathbf{b}})$ and a self-adjoint polynomial $P$ such that 
\[\|z_n-P(u(\tilde{a}), Y(\tilde{a}, \tilde{\mathbf b}))\|_{2, X}<1/2k.\]
Then it follows from Lemma~\ref{lem:repair} and continuity that there exists a $a_2\in \mathcal S$ such that $\|\tilde{a}-a_2\|<r/2$ and $\|\tilde{a}-a_2\|_{2, T(A)}$ small enough making $\|P(u(a_2), Y(u(a_2),\tilde{\mathbf{b}}))-P(u(\tilde{a}), Y(\tilde{a}, \tilde{\mathbf b}))\|<1/2k$. Therefore, this implies
\[\|z_n-P(u(a_2), Y(u(a_2),\tilde{\mathbf{b}}))\|_{2, T(A)}<1/k.\]

Combining these two cases, one has $U_{n,k}$ are dense open. Then Baire's theorem yields $(a,\mathbf b)\in \bigcap_{n,k}U_{n,k}$. Define
\[
 u=e^{ita},\quad \sigma(u)=\T,\qquad
 y=\sum_j\psi_j(u)b_j\psi_j(u),
\]
and write $B=C^*(1,u,y)$. For any $n, k$, there exists a polynomial $P_{n,k}$ such that
\[\|z_n-P_{n,k}(u,y)\|_{2,X}<1/k.\]
Now, let $z\in A_\sa$ with $\|z\|\leq 1$ and $\eta>0$. Choose $n, k$ large enough such that $\|z-z_n\|<\eta/2$ and $k>2/\eta$, and  thus
\[\|z-P_{n,k}(u,y)\|_{2, X}\leq \|z-z_n\|+\|z_n-P_{n,k}(u,y)\|_{2, X}<\eta.\]
This implies that $B$ is uniformly $2$-norm dense in $A$. Moreover, define $\chi(s)=\max\{-1,\min\{s, 1\}\}$ and $c=\chi(P_{n,k}(u,y))\in B_\sa$ with $\|c\|\leq 1$. A calculation shows 
\[\|z-P_{n,k}(u,y)\|_{2,\rho}\leq\|z-c\|_{2,\rho}.\]
This means for any $z\in A^1_\sa$, and $\eta>0$, there exists a $c\in B^1_\sa$ such that $\|z-c\|_{2, X}<\eta$.

Finally, let $f\in C(\T)_+\setminus\{0\}$. Since $\sigma(u)=\T$, the element $f(u)\neq 0$ and $\rho(f(u))>0$ for any $\rho\in X$. Because $X$ is compact, this implies that
\[
 \inf_{\rho\in X}\rho(f(u))>0.
\]
Together with the strict comparison for positive elements, we have verified all hypotheses of \cite[Theorem 5.1]{ELN}. That theorem gives
$A=C^*(u(R1+h))$ for some self-adjoint $h\in A$ and $R>\norm h$.
\end{proof}

\section{Acknowledgement}
The author would like to thank Yongle Jiang for bringing his attention to the generator problem for $C^*$-algebras and several helpful discussions. Some partial results was obtained before the announcement of \cite{SG}. During  the development of this work, the author used  GPT 6 Astra for searching necessary references, examining mathematical arguments, and language editing.  The authors have checked and substantially revised all mathematical arguments generated by AI and take full responsibility for the final version.


\begin{thebibliography}{TC}
\bibitem{AGEP}T. Amrutam, D. Gao, S. K. Elayavalli, G. Patchell. Strict comparison in reduced group C*-algebras. \textit{Invent. Math.} \textbf{242} (2025), 639-657.


\bibitem{BCC}B. Bouwen, J. Castillejos, and J. Curda. The generator problem fror $C^*$-algebras  with comparison. arXiv:2610.13792.

\bibitem{TC}
J.~R.~Carri\'on, J.~Castillejos, S.~Evington, J.~Gabe,
C.~Schafhauser, A.~Tikuisis, and S.~White.
Tracially complete $C^*$-algebras. To appear in \textit{Mem. Amer. Math. Soc.}. arXiv:2310.20594.

\bibitem{ELN1}G.~A.~Elliott, C.~G.~Li, and Z.~Niu. A simple spearable $C^*$-algebra which is not singly generated. arXiv:2608.13707.

\bibitem{ELN}
G.~A.~Elliott, C.~G.~Li, and Z.~Niu.
On the generators of separable unital $C^*$-algebras. arXiv:2610.17411.

\bibitem{LNR}
C.~G.~Li, Z.~Niu and V. M. Ruzicka. Villadsen algebras are singly generated. To appear in \textit{J. Noncommut. Geom.} arXiv: 2604.07688.

\bibitem{LS}W. H. Li and J. H. Shen. A note on approximately divisible $C^*$-algebras. arXiv:0804.0465.

\bibitem{MWY}X. Ma, D. Wang, and W. Yang. Boundary actions of Bass-Serre trees and applications to $C^*$-algebras. To appear in \textit{J. Noncommut. Geom.}. arXiv: 2502.02039.

\bibitem{OZ}C. L. Olsen and W. R. Zame. Some $C^*$-algebras with a single generator. \textit{Trans. Amer. Math. Soc.} \textbf{215} (1976), 205-217.



\bibitem{SG}
OpenAI. Relative generation and the generator problem for finite
factors. September 23, 2026.
\href{https://github.com/openai/math/blob/main/preprints/Relative-generation-and-the-generator-problem-for-finite-factors-September-23-2026/paper.pdf}{Manuscript}.

\bibitem{O}N. Ozawa. Proximality and selflessness for group $C^*$-algebras. arXiv: 2508.07938.

\bibitem{R}L. Robert. Selfless $C^*$-algebras. \textit{Adv. Math.} \textbf{478}(2025), 110409.

\bibitem{STW}
C.~Schafhauser, A.~Tikuisis, and S.~White.
Nuclear $C^*$-algebras: 99 problems. arXiv:2506.10902v2.



\bibitem{TW}
H. Thiel and W. Winter. The generator problem for $\mathcal Z$-stable $C^*$-algebras. \textit{Trans. Amer. Math. Soc.} \textbf{366} (2014), 2327-2343.

\bibitem{T}D. M. Topping. UHF algebras are singly generated. \textit{Math. Scand.} \textbf{22} (1968), 224-226.

\bibitem{Y}W. Yang. An extreme boundary of acylindrically hyperbolic groups. arXiv:2511.16400.
\end{thebibliography}
\end{document}